\documentclass[12pt,leqno]{amsart}
\usepackage{amsmath,amsfonts,enumerate}
\usepackage{graphicx}
\usepackage{xcolor}

\newcounter{fig}
\newtheorem{theorem}{Theorem}[section]

\newtheorem{definition}[theorem]{Definition}

\newtheorem{lemma}[theorem]{Lemma}

\newtheorem{proposition}[theorem]{Proposition}

\newtheorem{remark}{Remark}[section]
\numberwithin{equation}{section}

\newcommand{\baa}{\begin{array}}
\newcommand{\eaa}{\end{array}}

\def\R{\mathbb{R}}

\def\epsilon{\varepsilon}

\newcommand{\ld}{\lambda}
\newcommand{\mU}{\mathcal{U}}
\newcommand{\mL}{\mathcal{L}}
\newcommand{\eldo}{e^{\ld_1 z}}
\newcommand{\elds}{e^{\ld_u z}}
\newcommand{\eldt}{e^{\ld_3 z}}
\newcommand{\eldss}{e^{\ld^* z}}

\newcommand\ophi{{\overline{\phi}}}
\newcommand\uphi{{\underline{\phi}}}
\newcommand\ep{\varepsilon}

\newcommand{\bi}{\begin{itemize}}
\newcommand{\ei}{\end{itemize}}
\newcommand{\ben}{\begin{enumerate}}
\newcommand{\een}{\end{enumerate}}

\newcommand\be{\begin{equation}}
\newcommand\ee{\end{equation}}
\newcommand\bea{\begin{eqnarray}}
\newcommand\eea{\end{eqnarray}}
\newcommand\beaa{\begin{eqnarray*}}
\newcommand\eeaa{\end{eqnarray*}}
\newcommand\bss{\begin{cases}}
\newcommand\ess{\end{cases}}
\newcommand\triphi{{(\phi_1,\phi_2,\phi_3)}}

\begin{document}
\title[Forced waves]{Forced waves of a three species predator-prey system with a pair of weak-strong competing preys in a shifting environment}

\author[T. Giletti]{Thomas Giletti}
\address{Institut Elie Cartan de Lorraine, UMR 7502, University of Lorraine, 54506 Vandoeuvre-l\`es-Nancy, France}
\email{thomas.giletti@univ-lorraine.fr}

\author[J.-S. Guo]{Jong-Shenq Guo}
\address{Department of Mathematics, Tamkang University, Tamsui, New Taipei City 251301, Taiwan}
\email{jsguo@mail.tku.edu.tw}

\thanks{Date: \today. Corresponding Author: J.-S. Guo.}
\thanks{This work was partially supported by the Ministry of Science and Technology of Taiwan under the grant 108-2115-M-032-006-MY3, and by the CNRS-NCTS joint International Research Network ReaDiNet.
 We would like to thank the anonymous referees for the careful readings and some valuable comments.}



\begin{abstract}
In this paper, we investigate so-called forced wave solutions of a three components reaction-diffusion system from population dynamics.
Our system involves three species that are respectively two competing preys and one predator; moreover, the competition between both preys is strong,
i.e. in the absence of the predator, one prey is driven to extinction and the other survives.
Furthermore, our problem includes a spatio-temporal heterogeneity in a moving variable that typically stands as a model for climate shift.
In this context, forced waves are special stationary solutions which are expected to describe the large-time behavior of solutions,
and in particular to provide criteria on the climate shift speed to allow survival of either of the three species.
We will consider several types of forced waves to deal with various situations depending on which species are indigenous and which species are aboriginal.
\end{abstract}

\maketitle
\setlength{\baselineskip}{18pt}
\section{Introduction}
\setcounter{equation}{0}

In this paper, we study the following three-species predator-prey system with a pair of weak-strong competing preys
\be\label{pde}
\begin{cases}
u_t = d_1u_{xx} + r_1u[1+\alpha(x+st)-u-kv-b_1w],\;x\in\R,\,t>0,\\
v_t = d_2v_{xx} + r_2v[1+\alpha(x+st)-hu-v-b_2w],\;x\in\R,\,t>0,\\
w_t = d_3w_{xx} + r_3w[-1+\alpha(x+st)+a_1u+a_2v-w],\;x\in\R,\,t>0,
\end{cases}
\ee
in which $u,v$ stand for the densities of two competing preys, $w$ is the density of the predator, and all parameters in system~\eqref{pde} are positive.

The function $\alpha$ models the effect of climate change and we assume it to be a negative bounded continuous function such that
\be\label{a1}
 \alpha(z)\ge -Ce^{\rho z},\;\forall\, z\le -K,\; \alpha(z)<-1,\;\forall\, z\geq K,\;\mbox{for some positive constants $C,\rho,K$.}
\ee
The positive constant $s$ is the shifting speed of the environment, and we typically assume it to be positive. Therefore the favorable habitat of all three species is receding (in the left direction) as time advances.
The $d_i$, $i=1,2,3$, denote the diffusion coefficients. The (heterogeneous) functions $[1+\alpha(x+st)]r_i$, $i=1,2$, are the intrinsic growth rates of $u,v$, respectively,
and the net growth rate of the predator $w$ is assumed to be $[-1+\alpha(x+st)]r_3$. In particular, due to $\alpha \leq 0$, it is assumed that the predator cannot survive without the feeding of preys.
Moreover, the two preys obey the logistic growth rule. Regarding coupling terms, $a_i,b_i$, $i=1,2$, are the conversion rates and the predation rates for the preys, respectively,
and $h,k$ are inter-specific competition coefficients {such that $0<h<1<k$. This means that $u$ is a weak competitor and $v$ is a strong competitor, at least in the absence of the predator.}

For simplicity, in this paper we only consider the case with equal diffusions, predation rates and conversion rates, namely,
\beaa
d_1=d_2=d_3:=d,\quad a_1=a_2:=a,\quad b_1=b_2:=b.
\eeaa
Moreover, we always assume that
\be\label{c-s}
{a>1,\quad 0<h<1<k,}
\ee
so that the predator can survive with the feeding of the competing preys.


We are mainly concerned with the existence and non-existence of forced waves, namely, traveling waves which move with the same speed $s >0$ as the environmental shift.
More precisely, a forced wave of~\eqref{pde} is an entire in time and bounded solution of the form
$$(u,v,w)(x,t):=(\phi_1,\phi_2,\phi_3)(z),\quad z:=x+st,$$
for some function $(\phi_1,\phi_2,\phi_3)$ (which we call the {\it wave profiles}).
Hence $(\phi_1,\phi_2,\phi_3)$ satisfies
\be\label{TWS}
\begin{cases}
d\phi_1''(z)-s\phi_1'(z)+r_1\phi_1(z)[\alpha(z)+(1-\phi_1-k\phi_2-b\phi_3)(z)]=0,\\
d\phi_2''(z)-s\phi_2'(z)+r_2\phi_2(z)[\alpha(z)+(1-h\phi_1-\phi_2-b\phi_3)(z)]=0,\\
d\phi_3''(z)-s\phi_3'(z)+r_3\phi_3(z)[\alpha(z)+(-1+a\phi_1+a\phi_2-\phi_3)(z)]=0,
\end{cases}
\ee
for all $z\in\R$.


According to condition~\eqref{a1}, the space-time environmental heterogeneity is such that the favorable region to the three species decreases as time increases.
In particular, in the non-moving frame all species go to extinction. Thus a forced wave must satisfy that
\be\label{rl}
\triphi(\infty)=(0,0,0),
\ee
which is the only non-negative steady state of the limiting problem as $x + st \to +\infty$.
{The condition \eqref{rl} can be verified in the same manner as that of \cite[Proposition 2.2]{cg22}.
It is simply due to the fact that at $+\infty$ the growth rate for both preys is negative, which ultimately drive them as well as the predator to extinction. We refer the reader to \cite{cg22} and omit its proof here.}

On the other hand, there are several possible limiting states at $z=-\infty$.
Since the forced wave moves to the left, the choice of this limiting state may be interpreted as that of the initial condition before the climate change.
Computing the nontrivial constant states of \eqref{TWS} without the $\alpha$ term, we find the possible limiting states
\be\label{non}
E_u:=(1,0,0),\quad E_v:=(0,1,0),\quad E^*:=(u_p,0,w_p),\quad E_*:=(0,v_p,w_p),
\ee
where
\beaa
v_p=u_p := \frac{1+b}{1+ab},\quad w_p := \frac{a-1}{1+ab}.
\eeaa
Due to our assumption that $h < 1 < k$, both competing preys cannot co-exist and thus there is no positive co-existence state for system \eqref{TWS} without $\alpha$ term.
In particular, the above formulae for the steady states do not involve the competition parameters $h$ and $k$.

When the climate change effect is taken into account, there have been a lot of studies on the existence of forced waves for
the scalar equation and two-component competition (or cooperative) systems.
See, e.g., \cite{br08,br09,bf18,bg19,flw16,hz17,lwz18} and \cite{co1,ch20,dsz20,wz19} respectively for scalar local and nonlocal equations, and \cite{ywl19,ww21,dll21} for two species models.
All models in these works satisfy the comparison principle so that the classical monotone iteration method can be applied.
With the shifting effect, little is done for two or three species predator-prey systems where such an approach is not available.

On the other hand, an application of Schauder's fixed point theorem with the help of so-called generalized upper-lower solutions
has been proved to be very successful in many homogeneous predator-prey systems (without the shifting effect) to study traveling wave solutions.
In this aspect, we refer the reader to \cite{M01,WZ01,HLR03,HZ03,LLR06,LLM10,H12,L14,LR14,CGY,ZJ17} for 2-species cases and
to \cite{DX12,HL14,L15,Z17,L18,BP18,GNOW20,CGG21,CG21} for 3-species cases. Recently, this method was used  to derive the existence of forced waves
in the work \cite{choi21} for a two species predator-prey system in a shifting environment and
in \cite{cg22} for system \eqref{pde} with two weak competing preys (i.e. $h,k<1$) and with non-equal diffusion.
Note that in \cite{choi21,cg22} the monotonicity of $\alpha$ is imposed, but the shifting effect on the predator is only indirect.
In this work, we shall remove the assumption on the monotonicity of $\alpha$ imposed in \cite{choi21,cg22}.
Moreover, we impose directly the shifting effect in the predator equation of~\eqref{pde}.




The rest of this paper is organized as follows.
In Section~\ref{sec:main}, we state our main results on the existence and non-existence of forced waves.
We shall deal separately with different types of forced waves depending on the invaded state $(\phi_1, \phi_2, \phi_3) (-\infty)$.
Section~\ref{sec:profiles} is concerned with the proofs of the existence results by the method of generalized upper-lower solutions.
The formulas of generalized upper-lower solutions for forced waves with invaded states $E_u$ and $E^*$ are exactly the same as that in \cite{cg22}.
However, the conditions we impose here are different from that in \cite{cg22} (see Remark~\ref{rk1} below),
to be self-contained and for the reader's convenience we give the detailed construction of upper-lower solutions here.
The verification of these upper-lower solutions and the technical computations this involves shall be postponed to the Appendix.
Lastly, in Section~\ref{sec:profiles} we derive the existence of forced waves with invaded state~$E_*$ for any $s>0$.
In particular, we shall include a very simple proof of the existence of a positive forced wave for scalar equations without any monotonicity assumption on the shifting heterogeneity~$\alpha$.
See Proposition~\ref{scalar} below.

\section{Main results}\label{sec:main}

In this section, we proceed to the statements of our main results on sufficient and necessary conditions for the existence of forced waves. First note that, by a suitable translation, condition~\eqref{a1} can be rephrased as
\be\label{a2}
 \alpha(z)\ge -\ep e^{\rho z},\quad \forall\, z<0,
\ee
for any given positive constant $\ep$.
Indeed, for a given positive constant $\ep$, from \eqref{a1} we have
\beaa
 \theta(z):=\alpha(z-M)\ge -Ce^{-\rho M}e^{\rho z}\ge -\ep e^{\rho z},\quad  \forall\, z<0,
\eeaa
if we choose the constant $M\ge K$ large enough such that $Ce^{-\rho M}\le\ep$.
Replacing $\alpha(z)$ in system \eqref{TWS} by $\theta(z)$, without loss of generality and for convenience,
we shall assume~\eqref{a2} in the rest of this paper for some $\ep >0$ that can be made arbitrarily small as needed.
Note that a solution $\triphi(z)$ of the translated system \eqref{TWS} with $\alpha(z)=\theta(z)$ renders a solution $\triphi(z+M)$ of the original system \eqref{TWS} with $\alpha(z)$.

In the sequel we denote
\beaa
\beta^*:=1-hu_p-bw_p,\qquad \beta_*:=1-kv_p-bw_p.
\eeaa
Then one can check that $\beta^*>0$ and $\beta_*<0$, using $h<1$ and $k>1$, since
\beaa
 \beta^*=(1-h)(1+b)/(1+ab),\quad \beta_*=(1-k)(1+b)/(1+ab).
\eeaa
Moreover, these are the growth rates of respectively $v$ around the steady state $E^*$, and $u$ around the steady state $E_*$, when $\alpha$ is replaced by~$0$.
Hence $E^*$ is unstable and $E_*$ is stable (in the ODE sense and without the $\alpha$ term).
Note that $E_u$ and $E_v$ are always unstable. Next, we also let
\beaa
s_2^*:=2\sqrt{dr_2(1-h)},\quad  s_2^{**}:=2\sqrt{dr_2\beta^*},\quad s_3^*:=2\sqrt{dr_3(a-1)},
\eeaa
which denote respectively the minimal traveling wave speeds in the homogeneous equations
$$
\left.
\begin{array}{l}
v_t = d v_{xx} + r_2 v (1-h -v), \vspace{5pt}\\
v_t = d v_{xx} + r_2 v (\beta^* -v), \vspace{5pt}\\
w_t = d w_{xx} + r_3 w (a-1 -w).
\end{array}
\right.
$$
In other words, at least formally $s_2^*$ is the minimal speed of invasion of the strong prey in a homogeneous environment populated only by its weaker competitor,
while $s_2^{**}$ is the minimal speed of invasion of the strong prey in a homogeneous environment populated by both the weak competitor and the predator.
Similarly, $s_3^*$ may be understood as the minimal invasion speed of the predator in an environment populated by only one of either prey.

{By analogy with the scalar equation, it is natural to expect that the existence or not of forced waves of~\eqref{pde} is related to the value of the shifting speed $s$
and how it compares to the invasion speeds~$s_2^*$, $s_3^*$ and $s_2^{**}$}. This is already confirmed by our first main result which provides necessary conditions for the existence of (positive) forced waves.
\begin{theorem}\label{th:nec}
Under conditions \eqref{a1} and \eqref{c-s}, there exists a positive solution $\triphi$ of \eqref{TWS}
\bea
&&\mbox{\,with $\triphi(-\infty)=(1,0,0)$ only if $s\ge\max\{s_2^*,s_3^*\}$;}\label{eu}\\
&&\mbox{\,with $\triphi(-\infty)=(0,1,0)$ only if $s\ge s_3^*$;}\label{ev}\\
&&\mbox{\,with $\triphi(-\infty)=(u_p,0,w_p)$ only if $s\ge s_2^{**}$.}\label{estar}
\eea
\end{theorem}
{Since the proof of Theorem~\ref{th:nec} is exactly the same as that of \cite[Proposition 4.3]{cg22}, we omit it here. It relies on estimating the exponential convergence of the forced wave to its limiting steady state.
See also \cite[Proposition 4.9]{cg22}.
Hereafter $\triphi$ is positive means $\phi_i>0$ in $\R$ for all $i=1,2,3$.}
We remark that condition \eqref{rl} is not needed in Theorem~\ref{th:nec}, since \eqref{rl} is automatically satisfied for any positive solution of \eqref{TWS}.

In the first type of forced wave in Theorem~\ref{th:nec}, the invaded state is $(1,0,0)$.
Thus in these forced waves, a pulse of the second (strong competitor prey) and third (predator) species propagates together with the climate shift, into an environment populated by the first (weak competitor prey) species.
Similarly, in the second type of forced wave a pulse of the weak prey and of the predator is propagating into an environment inhabited by the strong prey;
in the third type, a pulse of the strong prey propagates into an environment where the weak prey and the predator cohabit.
Note that we omit in Theorem~\ref{th:nec} the case of a forced wave whose invaded state is $(0,v_p,w_p)$.
Indeed, $(0,v_p,w_p)$ is a stable steady state with respect to the underlying kinetic ODE system (that is, \eqref{pde} with $d_1 = d_2 = d_3 = 0$ and $\alpha \equiv 0$),
and therefore we expect that such a forced wave may exist for any positive speed; see Theorem~\ref{th:es} below.\medskip

Next, we seek conditions for the existence of these forced waves. We manage to show that, at least in some parameter ranges, the critical speeds in Theorem~\ref{th:nec} are optimal.
Our next three theorems each deal with a different type of traveling wave.

Let us first consider forced waves whose invaded steady state at $z = -\infty$ is $E_u$. We let
\beaa
Q_1(\rho):=
\left\{
\begin{array}{ll}
d\rho+r_3(a-1)/\rho, & \text{if }  \rho\in(0,\lambda_u),\\
s_3^*, & \text{if } \rho\ge \lambda_u:=\sqrt{r_3(a-1)/d}.
\end{array}
\right.
\eeaa
where~$\rho$ comes from~\eqref{a2}. Then we have the following result:
\begin{theorem}\label{th:eu}
Let conditions \eqref{a1} and \eqref{c-s} be enforced.
Suppose that $r_2(1-h)=r_3(a-1)$, $s\ge Q_1(\rho)$ and
\be\label{r1}
r_1[k+b(2a-1)-1]<r_3(a-1).
\ee
Then there exists a positive solution of \eqref{TWS} satisfying \eqref{rl} and
\be\label{llu}
\triphi( -\infty )=(1,0,0),
\ee
if $s\ge\max\{s_2^*,s_3^*\}=s_2^*=s_3^*$.
\end{theorem}

Such a forced wave corresponds to a situation where a ``pulse'' of the strong prey and the predator travel along the climate change in an environment which is initially populated only by the weak prey.
By an analogy with the homogeneous case, it is expected that such a forced wave can only be observed in the Cauchy problem if the initial data has a slow enough (typically exponential) decay as $x \to -\infty$.
If the initial populations of the strong prey $v_0$ and the predator $w_0$ are compactly supported, it is therefore expected that both species are driven to extinction.
In other words, this result formally suggests that $\max \{ s_2^* ,s_3^* \}$ is the maximal bearable climate shift speed in this situation.

\begin{remark}\label{rk1}
{It is left open whether there are forced waves invading $E_u$ when $s_2^*\neq s_3^*$. We conjecture that $\max \{s_2^*, s_3^*\}$ should be the minimal speed for forced waves.}

Also, for the forced waves invading $E_u$ for two weak competing preys in \cite{cg22}, we imposed the condition $0<b<\min\{1-h,1-k\}/(2a)$
which is void for the case of weak-strong competing preys here.
However, the same formula of upper-lower solutions constructed in~\cite{cg22} works for Theorem~\ref{th:eu} with condition \eqref{r1}.
\end{remark}


{Let us now turn to the situation where the invaded state is $E^*$, i.e. the weak prey~$u$ cohabits with the predator~$w$ ahead of the climate shift, and the strong prey~$v$ is the tentative invader.
In this case, we} let
\beaa
Q_2(\rho):=
\left\{
\begin{array}{ll}
d\rho+r_2\beta^*/\rho, & \text{if } \rho\in(0,\lambda^*),\\
s_2^{**}, & \text{if } \rho\ge \lambda^*:=\sqrt{r_2\beta^*/d}.
\end{array}
\right.
\eeaa
Then we have:
\begin{theorem}\label{th:estar}
Let conditions \eqref{a1} and \eqref{c-s} be enforced.
Suppose that $s\ge Q_2(\rho)$ and
\be\label{r12}
\max\{r_1[(k-1)+b(2a-1)],r_3\}<r_2\beta^*.
\ee
Then there exists a positive solution of \eqref{TWS} satisfying \eqref{rl} and
\be\label{llp}
\triphi( -\infty)=(u_p,0,w_p),
\ee
if $s\ge s_2^{**}$.
\end{theorem}

{Similarly as before, this highlights the value $s_2^{**}$ as the typical minimal speed for a positive forced wave into an environment populated only by the weak prey and the predator,
 (at least formally) as also the maximal climate change speed that is sustainable for an indigenous strong prey competitor to persist in such a situation.}

Lastly, for the stable state $E_*$, we have

\begin{theorem}\label{th:es}
Suppose that
\be\label{co-b22}
a > \frac{1}{1-h}, \quad b < \frac{1-h - 1/a}{2a-1}.
\ee
Then there is a solution $\triphi$ of \eqref{TWS} with $\phi_i>0$ in $\R$, $i=2,3$, such that
\beaa
\triphi(-\infty)=(0,v_p,w_p),\; \triphi(\infty)=(0,0,0),
\eeaa
for any $s>0$.
\end{theorem}
This shows that, if the strong competitor prey and the predator inhabits the whole (unbounded) favorable zone, they naturally always persist ahead of the climate shift, regardless of the introduction of the weak competitor prey.

Notice that in Theorem~\ref{th:es} we do not state that the first component $\phi_1$ is positive.
Actually this cannot be ensured in general and there is a parameter range where there does not exist a forced wave satisfying $(\phi_1 , \phi_2, \phi_3) (-\infty)  = (0,v_p,w_p)$ together with $\phi_1 >0$.
We shall discuss this in Section~\ref{sec:es} after the proof of Theorem~\ref{th:es}.

In Theorems 2.2 and 2.3, the invaded state is unstable which is why those forced waves had a positive minimal speed; as far as spreading is concerned,
these minimal speeds should be the maximal climate shift speed allowing the corresponding indigenous species to survive. Theorem 2.4 is different because the invaded state is stable;
such a forced wave represent the typical large time behavior of solutions of the Cauchy problem.
In particular, the fact that $\phi_1$ may be null suggests that the weak competing prey may be fully driven to extinction,
even in the moving frame of the climate shift where the populations of the strong prey and the predator are diminished.
\begin{remark}\label{rk2}
{We were unable to derive any nontrivial forced waves invading $E_v$. As a matter of fact, similarly to the case when the invaded state is $E_*$,
one may find a parameter range where a positive forced wave invading $E_v$ simply does not exist. See also Remark~\ref{rk2bis} below.
This indicates that a suitable pair of generalized upper-lower solutions is not always available and its construction is by no means trivial in general.}
\end{remark}


\section{Derivation of wave profiles}\label{sec:profiles}

This section is devoted to the existence of wave profiles by an application of Schauder's fixed point theorem.
To do so, we first introduce the notion of generalized upper-lower solutions as follows.

\begin{definition}
Given $s>0$.
{Continuous functions} $(\ophi_1,\ophi_2,\ophi_3)$ and $(\uphi_1,\uphi_2,\uphi_3)$ defined on~$\R$ are called a pair of generalized upper-lower solutions of \eqref{TWS}
if $\ophi_i''$, $\uphi_i''$, $\ophi_i'$, $\uphi_i'$, $i=1,2,3$, are bounded functions such that the following inequalities hold:
\bea
&& \mathcal{U}_1(z):=d\ophi_1''(z) -s\ophi_1'(z)+r_1\ophi_1(z)[1+\alpha(z)-\ophi_1(z)-k\uphi_2(z)-b\uphi_3(z)]\le 0,\label{u1}\\
&& \mathcal{U}_2(z):=d\ophi_2''(z) -s\ophi_2'(z)+r_2\ophi_2(z)[1+\alpha(z)-h\uphi_1(z)-\ophi_2(z)-b\uphi_3(z)]\le 0,\label{u2}\\
&& \mathcal{U}_3(z):=d\ophi_3''(z) -s\ophi_3'(z)+r_3\ophi_3(z)[-1+\alpha(z)+a\ophi_1(z)+a\ophi_2(z)-\ophi_3(z)]\le 0,\label{u3}\\
&& \mathcal{L}_1(z):=d\uphi_1''(z) -s\uphi_1'(z)+r_1\uphi_1(z)[1+\alpha(z)-\uphi_1(z)-k\ophi_2(z)-b\ophi_3(z)]\ge 0,\label{l1}\\
&& \mathcal{L}_2(z):=d\uphi_2''(z) -s\uphi_2'(z)+r_2\uphi_2(z)[1+\alpha(z)-h\ophi_1(z)-\uphi_2(z)-b\ophi_3(z)]\ge 0,\label{l2}\\
&& \mathcal{L}_3(z):=d\uphi_3''(z) -s\uphi_3'(z)+r_3\uphi_3(z)[-1+\alpha(z)+a\uphi_1(z)+a\uphi_2(z)-\uphi_3(z)]\ge 0,\label{l3}
\eea
for $z\in\mathbb{R}\setminus E$ with some finite set $E$.
\end{definition}

From such generalized upper-lower solutions, one can infer the existence of a solution $(\phi_1,\phi_2,\phi_3)$ to system~\eqref{TWS}, as stated in the following proposition.
\begin{proposition}\label{exist}
Given $s>0$. Suppose that system \eqref{TWS} has a pair of generalized upper-lower solutions $(\ophi_1,\ophi_2,\ophi_3)$ and $(\uphi_1,\uphi_2,\uphi_3)$ such that
\bea
&&\uphi_i(z)\le\ophi_i(z), \quad \forall\,z\in\R, \,  i=1,2,3,\label{order}\\
&& \lim_{\xi\searrow z}\ophi_i'(\xi)\le\lim_{\xi\nearrow z}\ophi_i'(\xi),\; \lim_{\xi\nearrow z}\uphi_i'(\xi)\le\lim_{\xi\searrow z}\uphi_i'(\xi), \quad \forall\,z\in E, \, i=1,2,3.\label{1st}
\eea
Then system \eqref{TWS} has a solution $(\phi_1,\phi_2,\phi_3)$ such that $\uphi_i\le\phi_i\le\ophi_i$, $i=1,2,3$.
\end{proposition}

The proof of Proposition~\ref{exist} is based 
on an application of Schauder's fixed point theorem. Its proof is by now standard, thus we safely omit it and refer the reader to, e.g., \cite{M01}.


With Proposition~\ref{exist} and \eqref{rl} in hand, the remaining task is to construct a suitable pair of generalized upper-lower solutions that capture the desired left-hand limit at $z=-\infty$
and satisfy conditions \eqref{order} and \eqref{1st}.
Note that, by the strong maximum principle for scalar equations, any nonnegative bounded solution $\triphi$ of \eqref{TWS} satisfies $\phi_i>0$ in $\R$ if $\phi_i\not\equiv 0$, for any $i\in\{1,2,3\}$.

In particular, in the next three subsections we give three different pairs of generalized upper-lower solutions to address the different types of forced waves according to their left-hand limit at $z=-\infty$.

\subsection{Case $E_u=(1,0,0)$}\label{subsec_Eu}\hspace{\fill} \medskip

In this subsection we construct a pair of generalized upper-lower solutions and prove Theorem~\ref{th:eu}.

First, we assume that $s>s_3^*$ and let
\beaa
A_1(\ld):=d\ld^2-s\ld+r_3(a-1).
\eeaa
Then, due to $r_2(1-h)=r_3(a-1)$, we also have
\beaa
A_1(\ld)=d\ld^2-s\ld+r_2(1-h),
\eeaa
and since $s > s_3^*$ there exist $0<\ld_1<\ld_2<\infty$ such that $A_1(\ld_i)=0$, $i=1,2$.
Note that $A_1(\ld)<0$ for all $\ld\in(\ld_1,\ld_2)$.

Let us now briefly check that \eqref{a1} holds with $\rho = \ld_1$. On the one hand, notice that $\lambda_u = \sqrt{\frac{r_3 (a-1)}{d}} \in (\lambda_1, \lambda_2)$.
Thus, if $\rho \geq \lambda_u$, then without any loss of generality we can reduce~$\rho$ in \eqref{a1} and assume that it holds with $\rho = \ld_1$. On the other hand, if $\rho < \lambda_u$, then
$$s \ge Q_1 (\rho) = d \rho +  \frac{r_3 (a-1)}{\rho},$$
hence
$$A_1 (\rho) = d \rho^2 - s \rho + r_3 (a-1) \leq 0.$$
We find that $\rho \in [\lambda_1, \lambda_2]$, and we conclude as announced that \eqref{a1} holds with~$\rho = \ld_1$.

Then, due to \eqref{r1}, we can choose $\ep$ such that
\be\label{ep1}
 0< \ep< \frac{r_3(a-1)-r_1[k+b(2a-1)-1]}{r_1 },
\ee
and, by a translation, \eqref{a2} holds for this $\ep$ and with $\rho=\ld_1$.

We now define
\be\label{up-lo-1}
\bss
\ophi_1(z)\equiv 1,\; \uphi_1(z)=\max\{1-e^{\ld_1 z}, 0\}, \vspace{3pt} \\
\ophi_2(z)=\min\{e^{\ld_1 z},1\},\; \uphi_2(z)=\max\{e^{\ld_1 z}-q_1e^{\mu_1 z},0\},\vspace{3pt} \\
\ophi_3(z)=(2a-1)\min\{e^{\ld_1 z},1\},\; \uphi_3(z)=\max\{(2a-1)e^{\ld_1 z}-q_2e^{\mu_2 z},0\},
\ess
\ee
where $\mu_i\in(\ld_1,\min\{\ld_2,2\ld_1\})$, $i=1,2$, and
\bea
&&q_1>\max\left\{1,\frac{r_2[\ep+1+b(2a-1)]}{-A_1(\mu_1)}\right\},\label{q1}\\
&&q_2>\max\left\{2a-1,\frac{r_3(2a-1)[\ep+(3a-1)]}{-A_1(\mu_2)}\right\}.\label{q2}
\eea
Then we have the following lemma, whose proof we postpone to the Appendix. 
\begin{lemma}\label{la:eu}
Let the assumptions of Theorem~\ref{th:eu} be enforced. Assume that $s>s_3^*$.
Then the functions $(\ophi_1,\ophi_2,\ophi_3)$ and $(\uphi_1,\uphi_2,\uphi_3)$ defined by \eqref{up-lo-1} are a pair of generalized upper-lower solutions of \eqref{TWS}
such that conditions \eqref{order} and \eqref{1st} hold.
\end{lemma}

Next, we assume that $s=s_3^*$. In this case, $A_1(\ld)=0$ has a double root $\ld_u  >0$.
We choose~$\ep$ such that
\be\label{ep2}
 0<\ep< \frac{e\{r_3(a-1)-r_1[k+b(2a-1) -1]\} }{r_1 }
\ee
and, similarly as above and up to a translation, we assume that~\eqref{a2} holds with this~$\ep$ and~$\rho=\ld_u$.

Set $B_0:=\ld_u e$, and define
\beaa
&&\ophi_1(z)\equiv 1,\; \uphi_1(z)=\bss
 1+B_0ze^{\ld_u z},\; z<-1/\ld_u,\\
0,\; z\ge -1/\ld_u,
\ess\\
&&\ophi_2(z)=\bss
 -B_0ze^{\ld_u z},\; z<-1/\ld_u,\\
1,\; z\ge -1/\ld_u,
\ess
\uphi_2(z)=\bss
 -B_0ze^{\ld_u z}-q_3\sqrt{|z|}e^{\ld_u z},\; z<z_3,\\
0,\; z\ge z_3,
\ess\\
&&\ophi_3(z)=\bss
 -(2a-1)B_0ze^{\ld_u z}, \; z<-1/\ld_u,\\
2a-1,\; z\ge -1/\ld_u,
\ess\\
&&\uphi_3(z)=\bss
(2a-1)  \left\{ -B_0z - q_4\sqrt{|z|} \right\} e^{\ld_u z},\; z<z_4,\\
0,\; z\ge z_4,
\ess
\eeaa
where $z_3:=-(q_3/B_0)^2$, $z_4:=-(q_4/B_0)^2$, $q_3>e\sqrt{\ld_u}$, $q_4>(2a-1)e\sqrt{\ld_u}$ and
\bea
&& q_3>4r_2\left(\frac{B_0}{d}\right)\left[\ep\left(\frac{5}{2B_0}\right)^{5/2}+[1+b(2a-1)]B_0\left(\frac{7}{2B_0}\right)^{7/2}\right], \label{q3}\\
&& q_4>4r_3 \left(\frac{B_0}{d}\right)\left[\ep\left(\frac{5}{2B_0}\right)^{5/2}+(a+ 1)B_0\left(\frac{7}{2B_0}\right)^{7/2}\right]. \label{q4}
\eea
Note that $z_i<-1/\ld_u$, $i=3,4$, and $s=s_3^*=2d\ld_u$.
Then we have:
\begin{lemma}\label{la:eue}
Let the assumptions of Theorem~\ref{th:eu} be enforced. Assume that $s=s_3^*$.
Then the functions $(\ophi_1,\ophi_2,\ophi_3)$ and $(\uphi_1,\uphi_2,\uphi_3)$ defined above are a pair of generalized upper-lower solutions of \eqref{TWS}
such that conditions \eqref{order} and \eqref{1st} hold.
\end{lemma}
We again postpone the proof to the Appendix. 
Note that in both cases $s> s_3^*$ and $s = s_3^*$, the generalized upper-lower solutions satisfy $(\ophi_1,\ophi_2,\ophi_3)(-\infty)=(\uphi_1,\uphi_2,\uphi_3)(-\infty)=(1,0,0)$.
Thus Theorem~\ref{th:eu} is proved by combining Lemmas~\ref{la:eu} and \ref{la:eue} with Proposition~\ref{exist} and \eqref{rl}.

\subsection{Case $E^*=(u_p,0,w_p)$}\label{subsec_E*}\hspace{\fill} \medskip

Let us now turn to the proof of Theorem~\ref{th:estar}. As in the previous section, we first consider the case $s>s_2^{**}$.
Set
\beaa
A_2(\ld):=d\ld^2-s\ld+r_2\beta^*.
\eeaa
Then $A_2(\ld)=0$ has two real roots $\ld_i$, $i=3,4$, with $0<\ld_3<\ld_4$ and $A_2(\ld)<0$ for all $\ld\in(\ld_3,\ld_4)$.
Thanks to \eqref{r12}, we can choose $\ep$ such that
\be\label{ep3}
 0<\ep<\min\{(r_2\beta^*/r_1)-[(k-1)+b(2a-1)],(r_2\beta^*/r_3)-1\}  .
\ee
Due to $s \geq  Q_2 (\rho)$, one can check that $\rho \geq \ld_3$, hence up to a translation and without loss of generality we can assume that~\eqref{a2} holds with this $\ep$ and $\rho=\ld_3$.

We construct
\be\label{up-lo-s1}
\bss
\ophi_1(z)=\min\{u_p+bw_pe^{\lambda_3 z},1\},\; \uphi_1(z)=\max\{u_p(1-e^{\lambda_3 z}), 0\}, \vspace{3pt} \\
\ophi_2(z)=\min\{e^{\lambda_3z},1\},\; \uphi_2(z)=\max\{e^{\lambda_3z}-\eta_1e^{\nu_1 z},0\},\vspace{3pt} \\
\ophi_3(z)=\min\{w_p+ B_1e^{\lambda_3 z},2a-1\},\; \uphi_3(z)=\max\{w_p(1-e^{\lambda_3 z}),0\},
\ess
\ee
where $B_1:=2a-1-w_p$, $\nu_1\in(\ld_3,\min\{\ld_4,2\ld_3\})$ and $\eta_1$ satisfies
\be\label{eta1}
\eta_1>\max\left\{1, \frac{ r_2[\ep+1+b(2a-1)]}{-A_2(\nu_1)}\right\}.
\ee
Then we have:
\begin{lemma}\label{la:estar}
Let the assumptions of Theorem~\ref{th:estar} be enforced. Assume that $s>s_2^{**}$.
Then the functions $(\ophi_1,\ophi_2,\ophi_3)$ and $(\uphi_1,\uphi_2,\uphi_3)$ defined by \eqref{up-lo-s1} are a pair of generalized upper-lower solutions of \eqref{TWS}
such that conditions \eqref{order} and \eqref{1st} hold.
\end{lemma}


Next, we assume that $s=s_2^{**}$.
In this case, $A_2(\ld)=0$ has a double root $\lambda^* = \sqrt{r_2 \beta^* /d}>0$.
We choose $\ep$ such that
\be\label{ep4}
 0<\ep<e\min\{(r_2\beta^*/r_1)-[(k-1)+b(2a-1)],(r_2\beta^*/r_3)-1\}
\ee
and, similarly as before, we can assume that~\eqref{a2} holds with this $\ep$ and $\rho=\ld^*$.

We define
\beaa
&&\ophi_1(z)=\bss
 u_p-b w_pB_2ze^{\ld^* z}, \; z<-1/\ld^*,\\
1,\; z\ge -1/\ld^*,
\ess
\uphi_1(z)=\bss
 u_p[1+B_2ze^{\ld^* z}], \; z<-1/\ld^*,\\
0,\; z\ge -1/\ld^*,
\ess \\
&&\ophi_2(z)=\bss
 -B_2ze^{\ld^* z}, \; z<-1/\ld^*,\\
1,\; z\ge -1/\ld^*,
\ess
\uphi_2(z)=\bss
 -B_2ze^{\ld^* z}-\eta_2\sqrt{|z|}e^{\ld^* z}, \; z<z_6,\\
0,\; z\ge z_6,
\ess \\
&&\ophi_3(z)=\bss
 w_p- B_1 B_2ze^{\ld^*  z}, \; z<-1/\ld^*,\\
2a-1,\; z\ge -1/\ld^*,
\ess
\uphi_3(z)=\bss
 w_p[1+B_2ze^{\ld^* z}], \; z<-1/\ld^*,\\
0,\; z\ge -1/\ld^*,
\ess
\eeaa
where $B_1=2a-1-w_p$, $B_2:=\ld^* e$, $z_6:=-(\eta_2/B_2)^2$ and $\eta_2$ satisfies
\be\label{eta2}
\eta_2>\max\left\{e\sqrt{\ld^*}, 4r_2\frac{B_2}{d}\left[\ep\left(\frac{5}{2B_2}\right)^{5/2}+\{1+b(2a-1)\}B_2\left(\frac{7}{2B_2}\right)^{7/2}\right]  \right\}.
\ee
Note that $z_6<-1/\ld^*$ and $s=s_2^{**}=2d\ld^*$. Then we have:
\begin{lemma}\label{la:estar2}
Let the assumptions of Theorem~\ref{th:estar} be enforced. Assume that $s=s_2^{**}$.
Then the functions $(\ophi_1,\ophi_2,\ophi_3)$ and $(\uphi_1,\uphi_2,\uphi_3)$ defined above are a pair of generalized upper-lower solutions of \eqref{TWS}
such that conditions \eqref{order} and \eqref{1st} hold.
\end{lemma}
The proofs of Lemma~\ref{la:estar} and~\ref{la:estar2} shall be dealt with in the Appendix. 
In either cases we have $(\ophi_1,\ophi_2,\ophi_3)(-\infty)=(\uphi_1,\uphi_2,\uphi_3)(-\infty)=(u_p,0,w_p)$,
so that Theorem~\ref{th:estar} immediately follows by combining Lemmas~\ref{la:estar} and \ref{la:estar2} with Proposition~\ref{exist} and \eqref{rl}.

\subsection{Case $E_*=(0,v_p,w_p)$}\label{sec:es}\hspace{\fill} \medskip

We now turn to last case of a forced wave satisfying $\triphi(-\infty)=E_*$, and throughout this subsection we fix $s >0$.
In particular we will discuss the positivity (or not) of such a forced wave at the end of this section.

Unlike~$E_u$ and~$E^*$, the steady state~$E_*$ is stable with respect to the underlying ODE system (i.e. without diffusion) without climate shift (i.e. $\alpha \equiv 0$).
Thus we use here a different approach by comparison with traveling waves of some scalar equations.
Our first step is to prepare the following proposition (see~\cite{bg19,hz17} for similar result under a monotonicity assumption on~$\alpha$).
\begin{proposition}\label{scalar}
Let $\widehat{\alpha}$ be a negative bounded continuous function such that there exist positive constants $C$, $\rho$ and $K$ with
\be\label{ha1}
 \widehat{\alpha}(z)\ge -Ce^{\rho z},\;\forall z\le -K;\quad \widehat{\alpha}(z)<-1,\;\forall z \geq K.
\ee
Given positive constants $d,s,r$ and {$0 < \gamma <- \limsup_{z \to +\infty} \widehat{\alpha} (z)$}, then there is a solution $\phi$ to
\bea
&&d\phi''(z)-s\phi'(z)+r\phi(z)[\gamma+\widehat{\alpha}(z)-\phi(z)]=0,\; z\in\R,\label{1d-eq}\\
&&\phi(-\infty)=\gamma,\quad \phi(\infty)=0,\label{1d-bc}
\eea
such that $0\le\phi\le \gamma$ in $\R$.
Moreover, there is a positive constant $\ld_0\in(0,\rho)$ such that
$$\phi(z)\ge \gamma(1-e^{\ld_0 z}),$$ for all $z<0$. 
\end{proposition}
\begin{proof}
To construct a solution of \eqref{1d-eq}-\eqref{1d-bc}, we use the monotone iteration method together with a pair of super-sub-solutions.
On one hand it is clear that $\ophi(z)\equiv \gamma  $ is a super-solution.

For the sub-solution, we first choose a positive constant $\ld_0<\rho$ small enough such that $d\ld_0^2-s\ld_0<0$.
Then we choose $\ep>0$ small enough such that
\be\label{ep}
d\ld_0^2-s\ld_0+r\ep<0.
\ee
Up to a translation, $\widehat{\alpha}$ satisfies \eqref{a2} with this chosen $\ep$.

Let $\uphi(z)=\gamma\max\{1-e^{\ld_0 z},0\}$. We compute, for $z < 0$,
\beaa
&&d\uphi''(z)-s\uphi'(z)+r\uphi(z)[\gamma+\widehat{\alpha}(z)-\uphi(z)]\\
&\ge& -\gamma e^{\ld_0 z}(d\ld_0^2-s\ld_0)+r\gamma (1-e^{\ld_0 z})[\gamma-\ep e^{\rho z}-\gamma+\gamma e^{\ld_0 z}]\\
&\ge& -\gamma e^{\ld_0 z}\{d\ld_0^2-s\ld_0 +r\ep e^{(\rho-\ld_0)z}\}\ge 0,
\eeaa
due to $\ld_0<\rho$ and \eqref{ep}.
Hence $\uphi$ is a sub-solution of \eqref{1d-eq}.

It follows from the monotone iteration method that a solution $\phi$ of \eqref{1d-eq} exists such that $\phi(-\infty)=\gamma$ and $\uphi\le\phi\le\ophi$ in $\R$.
The fact that $\phi(\infty)=0$ follows from the negativity of the reaction term $\gamma+\widehat{\alpha}(z)$ at $\infty$, by the same proof as that of \cite[Proposition 2.2]{cg22}.
Thus the proof of Proposition~\ref{scalar} is done.
\end{proof}

Next, we remark that any bounded nonnegative solution $\triphi$ of \eqref{TWS} must satisfy
$$
\begin{array}{l}
\underline{\phi}_1:= 0\le\phi_1\le \overline{\phi}_1 := 1, \vspace{3pt}\\
 0\le\phi_2 \le \overline{\phi}_2 := 1, \vspace{3pt}\\
 0\le\phi_3 \le \overline{\phi}_3 := 2a -1,
\end{array}
$$
in $\R$. This simply follows from the maximum principle for scalar equations, noticing for instance that
$$d \phi_1 '' - s \phi_1 ' + r_1 \phi_1 \left[ 1 - \phi_1 \right] \geq 0.$$


It follows from assumption~\eqref{co-b22} that $b(2a-1) <1-h$, and then from Proposition~\ref{scalar} that there is a solution $\phi=\uphi_2$ to
\be\label{co-2}
\bss
 d \phi''(z)-s\phi'(z)+r_2\phi(z)[\gamma_2+\alpha(z)-\phi(z)]=0,\; z\in\R.\\
\phi(-\infty)=\gamma_2:=1-h-b(2a-1)>0,\; \phi(\infty)=0,
\ess
\ee
such that $0\le\uphi_2\le\gamma_2<1$ and $\uphi_2 (z)\ge \gamma_2 (1-e^{\ld_0 z})$ for all $z<0$ for some constant $\lambda_0\in(0,\rho)$. 

Using again assumption~\eqref{co-b22} and Proposition~\ref{scalar} (with $\gamma = \gamma_3$ defined below, $\widehat{\alpha} = \alpha + a ( \uphi_2 - \gamma_2)$
 and the constant $\rho$ in \eqref{ha1} is replaced by $\ld_0$), we have a solution $\phi=\uphi_3$ to
\be\label{co-l3}
\bss
 d \phi''(z)-s\phi'(z)+r_3\phi(z)[-1+\alpha(z)+a\uphi_2(z)-\phi(z)]=0,\; z\in\R,\\
\phi(-\infty)= \gamma_3 : = -1+a\gamma_2  > 0,\;\phi(\infty)=0,
\ess
\ee
such that $0\le\uphi_3\le\gamma_3<a-1<2a-1$.


Then one can check that $(\ophi_1,\ophi_2,\ophi_3)$ and $(\uphi_1,\uphi_2,\uphi_3)$ are continuously differentiable such that \eqref{u1}-\eqref{l3} hold for all $z\in\R$.
Hence they are a pair of generalized upper-lower solutions of \eqref{TWS} such that \eqref{order} and \eqref{1st} hold.
It follows from Proposition~\ref{exist} that a solution $\triphi$ of \eqref{TWS} exists such that $\uphi_i\le\phi_i\le\ophi_i$, $i=1,2,3$.
Also, $\triphi(\infty)=(0,0,0)$, by \eqref{rl}.

{To derive $\triphi(-\infty)=E_*$, we use the method of contracting rectangles. We refer the reader to \cite[section 4.3.1]{CGG21} for the details.} This proves Theorem~\ref{th:es}.\\

We conclude this section by discussing the positivity of the forced wave $(\phi_1, \phi_2, \phi_3)$ connecting $(0,v_p,w_p)$ and $(0,0,0)$.
The positivity of the second and third component immediately follows from the strong maximum principle, yet it remains an open question whether $\phi_1  >0$.
Actually, one may check that $\phi_1 \equiv 0$ in some parameter range.

Indeed, consider $(\phi_1,\phi_2, \phi_3)$ constructed above. First, we have that $\phi_2 \geq \underline{\phi}_2$ and $\phi_3 \geq \underline{\phi}_3$, hence
$$ d \phi_1 '' - s \phi_1 ' + r_1 \phi_1 (1 + \alpha - \phi_1 - k \underline{\phi}_2 - b \underline{\phi}_3 ) \geq 0.$$
Furthermore, by construction the functions $\underline{\phi}_2$ and $\underline{\phi}_3$ are positive and independent of the parameter~$k$. They also have positive limits at $-\infty$,
while $\limsup_{z\to +\infty} \alpha(z)  < - 1$. Therefore, one can find $k $ large enough such that
$$1 + \alpha - \phi_1 - k \underline{\phi}_2 - b \underline{\phi}_3 < 0$$
for all $z \in \R$. Thus, the function $\phi_1$ cannot admit a positive maximum. Since $\phi_1 (\pm \infty) = 0$, we conclude that $\phi_1 \equiv 0$. As announced, the first component of the forced wave may not be positive.

{Notice that we carried out the above argument only for the forced wave that we constructed, which a priori may not be the unique forced wave connecting $(0,v_p,w_p)$ and $(0,0,0)$.
However, in Proposition~\ref{scalar} it is possible to choose $\phi$ as the minimal positive and bounded solution of \eqref{1d-eq}-\eqref{1d-bc}, which is stable from below.
Thus, one may check by comparison arguments with scalar equations that any forced wave $(\phi_1,\phi_2 ,\phi_3)$ still satisfies $\phi_2 \geq \underline{\phi}_2$ and $\phi_3 \geq \underline{\phi}_3$,
hence the above argument still applies.}


\begin{remark}\label{rk2bis}
{We point out that a similar argument can be made for forced waves $(\phi_1, \phi_2,\phi_3)$ whose invading state is $E_v =(0,1,0)$.
Indeed $\underline{\phi}_2$ above still acts as a subsolution for the $\phi_2$-equation, and putting this into the first equation, one can find $k$ large enough so that $1 +\alpha - k \uphi_2 < 0$, hence $\phi_1 \equiv 0$.}
\end{remark}

\bigskip
\section{Appendix: verifications of upper-lower solutions}\label{appendix}

In this section, we prove Lemmas~\ref{la:eu} to~\ref{la:estar2}. Since conditions \eqref{order} and \eqref{1st} clearly hold, we only verify conditions \eqref{u1}-\eqref{l3} for each case.
Moreover, since these conditions are also trivial in the subdomains where the generalized upper-lower solutions are constant,
it suffices to check \eqref{u1}-\eqref{l3} in the ranges where the generalized upper-lower solutions are non-constant.

\begin{proof}[Proof of Lemma~\ref{la:eu}]
It is clear that $\mU_1(z)\le 0$ for all $z\in\R$.

Next, for $z<0$ we compute
\beaa
\mU_2(z)\le e^{\ld_1 z}(d\ld_1^2-s\ld_1)+r_2e^{\ld_1 z}\{1-h+h\eldo-\eldo\}=r_2\eldo(h-1)\eldo\le 0,
\eeaa
using $\alpha<0$, $\uphi_3\ge 0$, $h<1$ and $A_1(\ld_1)=0$.

For $z<0$, we also have
\beaa
\mU_3(z) & \le & (2a-1)\eldo(d\ld_1^2-s\ld_1)+r_3(2a-1)\eldo\{-1+a+a\eldo-(2a-1)\eldo\}\\
&=&r_3(2a-1)\eldo(1-a)\eldo\le 0,
\eeaa
using $\alpha<0$, $A_1(\ld_1)=0$ and $a>1$.

Now we turn to the lower solutions. For $z<0$, we compute
\beaa
\mL_1(z)&\ge& -\eldo(d\ld_1^2- s \ld_1)+r_1(1-\eldo)\{-\ep\eldo+\eldo-k\eldo-b(2a-1)\eldo\}\\
&\ge&  \eldo\{ r_3(a-1) -r_1[\ep+k+b(2a-1)-1 ]\}\ge 0,
\eeaa
using~\eqref{a2} with $\rho = \lambda_1$ (see the discussion in Subsection~\ref{subsec_Eu}), $A_1(\ld_1)=0$ and \eqref{ep1}.

For $\uphi_2$, there is $z_1<0$ (due to $q_1>1$) such that $\uphi_2(z)=\eldo-q_1e^{\mu_1 z}$ for $z<z_1$ and $\uphi_2(z)=0$ for $z>z_1$.
Then we compute, for $z<z_1$,
\beaa
\mL_2(z)&\ge&\eldo(d\ld_1^2-s\ld_1)-q_1e^{\mu_1 z}(d\mu_1^2-s\mu_1)\\
&&+r_2(\eldo-q_1e^{\mu_1 z})\{ 1-h -\ep\eldo-(\eldo-q_1e^{\mu_1 z})-b(2a-1)\eldo\}\\
&\ge&e^{\mu_1 z}\{-q_1A_1(\mu_1)-r_2e^{(2\ld_1-\mu_1)z}[\ep+1+b(2a-1)]\}\ge 0,
\eeaa
where we used \eqref{a2} with $\rho = \lambda_1$, $A_1(\ld_1)=0$, $A_1(\mu_1)<0$, $\mu_1<2\ld_1$ and \eqref{q1}.

Lastly, there is $z_2<0$ (due to $q_2 > 2a-1$) such that $\uphi_3(z)=(2a-1)\eldo-q_2e^{\mu_2 z}$ for $z<z_2$ and $\uphi_3(z)=0$ for $z>z_2$.
Similarly, we have $\mL_3(z)\ge 0$ for $z<z_2$, by using the choice of $q_2$ in \eqref{q2}. This proves Lemma~\ref{la:eu}.
\end{proof}
\medskip
\begin{proof}[Proof of Lemma~\ref{la:eue}]
For $z<z_u:=-1/\ld_u$, we compute
\beaa
\mU_2(z)&\le& - B_0 z  \elds(d\ld_u^2-s\ld_u)-B_0\elds(2d\ld_u-s)\\
&&- r_2B_0 z \elds\{1-h- hB_0 z \elds + B_0 z \elds\} \\
& = &  -r_2B_ 0 z \elds\{- h B_0 z \elds + B_0 z \elds\} \le 0 ,
\eeaa
using $\alpha<0$, $s=2d\ld_u$, $\uphi_3\ge 0$, $h<1$ and $A_1(\ld_u)=0$.
We point out that $s=2d\ld_u$ comes from the fact that $\lambda_u$ is a double root of $A_1 (\lambda) = 0$, and this shall be used again below.

For the last upper solution, for $z<z_u$ we also have
 \beaa
\mU_3(z)\le - (2a-1)B_0 z \elds\{(d\ld_u^2-s\ld_u)+r_3[-1+a - (1-a)B_0 z \elds]\}\le0,
\eeaa
using $\alpha<0$, $s=2d\ld_u$, $a>1$ and $A_1(\ld_u)=0$.

Next, for $z<z_u$ we compute
 \beaa
\mL_1(z)& \ge & - B_0 z e^{\lambda_u z} (-d \lambda_u^2 + s \lambda_u) - B_0 e^{\lambda_u z} (-2d \lambda_u + s) - r_1 (1 + B_0 z e^{\lambda_u z}) \varepsilon e^{\lambda_u z} \\
& & - r_1 (1 + B_0 z e^{\lambda_u z} ) \{ 1 - k - b (2a -1 )\} B_0 z e^{\lambda_u z}   \\
&\ge& -B_0z\elds\{r_3(a-1)-r_1[\ep(-z B_0)^{-1} +k+b(2a-1) -1]\}\\
&  \ge & - B_0 z \elds\{r_3(a-1)-r_1[\ep/e+k+b(2a-1) -1 ]\}\ge 0,
\eeaa
using \eqref{a2} with $\rho = \ld_u$, $s=2d\ld_u$, $k>1$, $A_1(\ld_u)=0$, $B_0 =\lambda_u e = -e/z_u$ and \eqref{ep2}.

Moreover, for $z<z_3$, we have
 \beaa
\mL_2(z)&\ge&  \frac{1}{4}  (-z)^{-3/2} dq_3\elds \\
&&  - r_2 (- B_0 z - q_3 \sqrt{|z|})  \{ \varepsilon -  B_0  z - q_3 \sqrt{|z|} - b (2a -1)B_0  z  \} e^{2 \lambda_u z}\\
&\ge&  \frac{1}{4} (-z)^{-3/2} dq_3\elds + r_2B_0 z  e^{2 \lambda_u z} \{\ep - [1+b(2a-1)]B_0 z \}\\
&=&\frac{1}{4}(-z)^{-3/2}\elds\left\{dq_3- 4r_2B_0  [(-z)^{5/2}\ep  \elds+ (-z)^{7/2} B_0(1+b(2a-1))\elds]\right\},
\eeaa
using \eqref{a2} with $\rho = \lambda_u$, $s=2d\ld_u$, $h<1$ and $A_1(\ld_u)=0$.
Recall the fact that, for a given constant $\gamma>0$,
\beaa
(-z)^\gamma \elds\le  \left( \frac{\gamma}{e \lambda_u} \right)^\gamma = (\gamma/B_0)^\gamma,\quad \forall\, z<0.
\eeaa
It then follows from \eqref{q3} that $\mL_2(z)\ge 0$ for all $z<z_3 < 0$.

Finally, a similar calculation also leads $\mL_3(z)\ge 0$ for all $z<z_4$, using \eqref{q4}. The lemma is thus proved.
\end{proof}

\medskip
\begin{proof}[Proof of Lemma~\ref{la:estar}]
For $z<0$, we calculate
\beaa
\mU_1(z)\le bw_p\eldt(d\ld_3^2-s\ld_3)+r_1(u_p+bw_p\eldt)\{1-u_p-bw_p\eldt-bw_p+bw_p\eldt\}\le 0,
\eeaa
using $\alpha<0$, $\uphi_2\ge 0$, $1-u_p-bw_p=0$ and $A_2(\ld_3)=0$. 

Also for $z<0$, we compute
\beaa
\mU_2(z)\le \eldt(d\ld_3^2-s\ld_3)+r_2\eldt\{1-hu_p+hu_p\eldt-\eldt-bw_p+bw_p\eldt\}\le 0,
\eeaa
using $\alpha<0$, $A_2(\ld_3)=0$ and $-1+hu_p+bw_p=-\beta^*<0$, and
\beaa
\mU_3(z)\le B_1\eldt(d\ld_3^2-s\ld_3)+r_3\ophi_3(z)\{-1+au_p+abw_p\eldt+a\eldt-w_p-B_1\eldt\} \le 0,
\eeaa
using $A_2(\ld_3)=0$, $-1+au_p-w_p=0$ and the fact that $ab w_p + a - B_1 = 0$ by the definitions of $B_1$ and $w_p$.

Next, we compute, for $z<0$,
\beaa
\mL_1(z)&\ge& -u_p\eldt(d\ld_3^2-s\ld_3)+r_1u_p(1-\eldt)\{-\ep\eldt+u_p\eldt-k\eldt-bB_1\eldt\}\\
&=&-u_p\eldt(d\ld_3^2-s\ld_3)+r_1u_p(1-\eldt)\{-\ep\eldt-(k-1)\eldt-b(2a-1)\eldt\}\\
&\ge&u_p\eldt\{r_2\beta^*-r_1[\ep+(k-1)+b(2a-1)]\}\ge 0,
\eeaa
using \eqref{a2} with $\rho = \lambda_3$ (see the discussion in Subsection~\ref{subsec_E*}), $u_p+bw_p=1$, $A_2(\ld_3)=0$ and \eqref{ep3}.

Now, due to $\eta_1>1$, there is $z_5<0$ such that $\uphi_2(z)=\eldt-\eta_1 e^{\nu_1 z}$ for $z<z_5$ and $\uphi_2(z)=0$ for $z>z_5$.
Then we have, for $z<z_5<0$,
\beaa
\mL_2(z)&\ge&\eldt(d\ld_3^2-s\ld_3)-\eta_1 e^{\nu_1 z}(d\nu_1^2-s\nu_1)\\
&&+r_2(\eldt-\eta_1 e^{\nu_1 z})\{\beta^*-\ep\eldt-hbw_p\eldt-\eldt+\eta_1 e^{\nu_1 z}-bB_1\eldt\}\\
&\ge& e^{\nu_1 z}\{-\eta_1 A_2(\nu_1)-r_2[\ep+1+b(2a-1)]e^{(2\ld_3-\nu_1)z}\} \\
&& + r_2 \eta_1 e^{\nu_1 z} \{\ep\eldt + hbw_p\eldt +  \underline{\phi}_2 (z) + bB_1\eldt\}\\
&\ge& e^{\nu_1 z}\{-\eta_1 A_2(\nu_1)-r_2[\ep+1+b(2a-1)]e^{(2\ld_3-\nu_1)z}\}\ge 0
\eeaa
using $A_2(\ld_3)=0$, $B_1=2a-1-w_p$, $h<1$, $A_2(\nu_1)<0$, $\nu_1<2\ld_3$ and \eqref{eta1}.

Finally, for $z<0$, we compute in a similar way that
\beaa
\mL_3(z)\ge -w_p\eldt\{d\ld_3^2-s\ld_3+r_3(\ep+1)\} = - w_p e^{\lambda_3 z} \{ -r_2 \beta^* + r_3 (\varepsilon +1) \} \ge 0,
\eeaa
using \eqref{a2} with $\rho = \lambda_3$, $-1+au_p-w_p=0$, \eqref{ep3} and $A_2(\ld_3)=0$.
This completes the proof of the lemma.
\end{proof}

\medskip
\begin{proof}[Proof of Lemma~\ref{la:estar2}]
Set $z^*:=-1/\ld^*$ for notational convenience.
For $z<z^*$, we have
 \beaa
\mU_1(z)&\le&- bw_pB_2 z \eldss(d(\ld^*)^2-s\ld^*)-bw_pB_2\eldss(2d\ld^*-s)\\
&&+r_1\ophi_1(z)\{1-u_p+ bw_pB_2 z \eldss-bw_p - bw_pB_2 z \eldss\}\le 0,
\eeaa
using $\alpha<0$, $\uphi_2\ge 0$, $s = 2d\ld^*$, $1-u_p-bw_p=0$ and $A_2(\ld^*)=0$. 

Next, for $z<z^*$ we compute
 \beaa
\mU_2(z)\le - r_2B_2 z \eldss \{\beta^*B_2 z \eldss\} \le 0,
\eeaa
using $\alpha<0$, $s=2d\ld^*$, $A_2(\ld^*)=0$ and $\beta^*=1-hu_p-bw_p>0$, as well as
 \beaa
\mU_3(z)\le - B_1B_2 z \eldss\{d (\ld^*)^2-s\ld^*\}\le 0,
\eeaa
using $\alpha<0$, $s=2d\ld^*$, $-1+au_p-w_p=0$, $A_2(\ld^*)=0$ 
and $B_1 = 2a - 1 - w_p =a (bw_p + 1)$ due to $w_p= (a-1)/(1+ab)$.

For the lower solutions, the computations proceed similarly as in the proof of Lemma~\ref{la:eue}, so that we only sketch these. First, we have for $z < z^*$ that
\beaa
 \mL_1(z)\ge - u_pB_2 z \eldss(r_2\beta^*)-r_1u_p\{\ep\eldss- [(k-1)+b(2a-1)] B_2 z\eldss\}\ge 0,
\eeaa
using \eqref{a2} with $\rho = \ld^*$, $s=2d\ld^*$, $A_2(\ld^*)=0$, $B_2 z = \lambda^* e z \leq -e $ for $z < z^* = - 1/\lambda^*$, 
\eqref{r12} and \eqref{ep4}.
Then, for $z<z_6$,
 \beaa
\mL_2(z) &\ge & \frac{1}{4}(-z)^{-3/2}\eldss\{d\eta_2-4r_2B_2[(-z)^{5/2}\ep \eldss+(1+hb w_p + b B_1) \times (-z)^{7/2} B_2\eldss ] \}\\
&\ge&  \frac{1}{4}(-z)^{-3/2}\eldss\{d\eta_2-4r_2B_2[(-z)^{5/2}\ep \eldss+(1+ b (2a-1)) \times (-z)^{7/2} B_2\eldss ] \}\ge 0,
\eeaa
using \eqref{a2} with $\rho = \ld^*$, $s=2d\ld^*$, $A_2(\ld^*)=0$, $h <1$ and \eqref{eta2}.
Finally, $\mL_3(z)\ge 0$ for $z<z^*$ follows in the same manner as before (using \eqref{r12} and \eqref{ep4}).
This concludes the proof of Lemma~\ref{la:estar2}.
\end{proof}


\begin{thebibliography}{99}


\bibitem{br08}
H. Berestycki, L. Rossi, {\it Reaction-diffusion equations for population dynamics with forced speed, I - the case of the whole space}, Discrete Contin. Dyn. Syst., 21 (2008), 41-67.

\bibitem{br09}
H. Berestycki, L. Rossi, {\it Reaction-diffusion equations for population dynamics with forced speed, II - cylindrical type domains}, Discrete Contin. Dyn. Syst., 25 (2009), 19-61.

\bibitem{bf18}
H. Berestycki, J. Fang,
{\it Forced waves of the Fisher-KPP equation in a shifting environment},
J. Differential Equations, 264 (2018), 2157-2183.


\bibitem{bg19}
J. Bouhours, T. Giletti,
{\it Spreading and vanishing for a monostable reaction-diffusion equation with forced speed}, J. Dynam. Differential Equations, 31 (2019), 247-286.

\bibitem{BP18}
Z. Bi and S. Pan,
{\it Dynamics of a predator-prey system with three species}, {Bound. Value. Probl.} (2018) 2018:162.



\bibitem{CGG21}
Y.-S. Chen, T. Giletti and J.-S. Guo,
{\it Persistence of preys in a diffusive three species predator-prey system with a pair of strong-weak competing preys}, J. Differential Equations 281 (2021), 341-378.

\bibitem{CG21}
Y.-S. Chen and J.-S. Guo,
{\it Traveling wave solutions for a three-species predator-prey model with two aborigine preys},
Japan J. Industrial and Applied Mathematics 38 (2021), 455-471. 

\bibitem{CGY}
Y.-Y. Chen, J.-S. Guo and C.-H. Yao,
{\it Traveling wave solutions for a continuous and discrete diffusive predator-prey model}, {J. Math. Anal. Appl.} {445} (2017), 212-239.

\bibitem{choi21}
W. Choi, T. Giletti, J.-S. Guo,
{\it Persistence of species in a predator-prey system with climate change and either nonlocal or local dispersal},
J. Differential Equations, 302 (2021), 807-853.

\bibitem{cg22}
W. Choi, J.-S. Guo,
{\it Forced waves of a three species predator-prey system in a shifting environment}, J. Math. Anal. Appl., 514 (2022), 126283. 


\bibitem{co1}
J. Coville,
{\it Can a population survive in a shifting environment using non-local dispersion}, arXiv:2012.09441.

\bibitem{ch20}
J. Coville, F. Hamel,
{\it On generalized principal eigenvalues of nonlocal operators with a drift}, Nonlinear Analysis, 193 (2020), Art. 111569, 20 pp.




\bibitem{dsz20}
P. De Leenheer, W. Shen, A. Zhang,
{\it Persistence and extinction of nonlocal dispersal evolution equations in moving habitats}, Nonlinear Analysis: Real World Appl., 54 (2020), Art. 103110, 33 pp.

\bibitem{dll21}
F.-D. Dong, B. Li, W.-T. Li,
{\it Forced waves in a Lotka-Volterra competition-diffusion model with a shifting habitat}, J. Differential Equations, 276 (2021), 433-459.

\bibitem{DX12}
Y. Du and R. Xu,
{\it Traveling wave solutions in a three-species food-chain model with diffusion and delays}, {Int. J. Biomath.} 5 (2012), 1250002, 17 pp.

\bibitem{flw16}
J. Fang, Y. Lou, J. Wu,
{\it Can pathogen spread keep pace with its host invasion?}, SIAM J. Appl. Math., 76 (2016), 1633-1657.


\bibitem{GNOW20}
J.-S. Guo, K.-I. Nakamura, T. Ogiwara and C.-C. Wu,
{\it Traveling wave solutions for a predator-prey system with two predators and one prey},
{Nonlinear Analysis: Real World Applications} {54} (2020), 103111, 13 pp.

\bibitem{hz17}
H. Hu, X. Zou,
{\it Existence of an extinction wave in the Fisher equation with a shifting habitat}, Proc. Amer. Math. Soc., 145 (2017), 4763-4771.

\bibitem{HLR03}
J. Huang, G. Lu and S. Ruan,
{\it Existence of traveling wave solutions in a diffusive predator-prey model}, J. Math. Biol. 46 (2003), 132-152.

\bibitem{H12}
W. Huang,
{\it Traveling wave solutions for a class of predator-prey systems}, J. Dynam. Differential Equations 24 (2012), 633-644.

\bibitem{HL14}
Y.L. Huang and G. Lin,
{\it Traveling wave solutions in a diffusive system with two preys and one predator}, {J. Math. Anal. Appl.} {418} (2014), 163-184.

\bibitem{HZ03}
J. Huang and X. Zou,
{\it Existence of traveling wave fronts of delayed reaction-diffusion systems without monotonicity}, {Disc. Cont. Dyn. Systems} 9 (2003), 925-936.

\bibitem{LLR06}
 W.T. Li, G. Lin and S. Ruan,
 {\it Existence of traveling wave solutions in delayed reaction-diffusion systems with applications to diffusion-competition systems},
{ Nonlinearity} 19 (2006), 1253-1273.

\bibitem{lwz18} W.-T. Li, J.-B. Wang, X.-Q. Zhao,
{\it Spatial dynamics of a nonlocal dispersal population model in a shifting environment}, {J. Nonlinear Sci.}, 28 (2018), 1189-1219.

\bibitem{L14}
 G. Lin, {\it Invasion traveling wave solutions of a predator-prey system}, {Nonlinear Anal.} 96 (2014), 47-58.

\bibitem{LLM10}
 G. Lin, W.T. Li and M. Ma,
{\it Traveling wave solutions in delayed reaction diffusion systems with applications to multi-species models},
{ Disc. Cont. Dyn. Systems, Ser. B} 13 (2010), 393-414.

\bibitem{LR14}
 G. Lin and S. Ruan,
{\it Traveling wave solutions for delayed reaction-diffusion systems and applications to diffusive Lotka-Volterra competition models with distributed delays},
{ J. Dyn. Diff. Equat.}, 26 (2014), 583-605.

\bibitem{L15}
J.-J. Lin, W. Wang, C. Zhao and T.-H. Yang,
{\it Global dynamics and traveling wave solutions of two predators-one prey models},
{Discrete and Continuous Dynamical System, Series B}, {20} (2015), 1135-1154.

\bibitem{L18}
J.-J. Lin and T.-H. Yang,
{\it Traveling wave solutions for a diffusive three-species intraguild predation model}, {Int. J. Biomath.} {11} (2018), 1850022, 27 pp.

\bibitem{M01}
S. Ma,
{\it Traveling wavefronts for delayed reaction-diffusion systems via a fixed point theorem},
{J. Differential Equations} {171} (2001), 294-314.

\bibitem{ww21} J.-B. Wang, C. Wu,
{\it Forced waves and gap formations for a Lotka-Volterra competition model with nonlocal dispersal and shifting habitats}, Nonlinear Analysis: Real World Applications, 58 (2021), 103208.

\bibitem{wz19}
J.-B. Wang, X.-Q. Zhao,
{\it Uniqueness and global stability of forced waves in a shifting environment}, Proc. Amer. Math. Soc., 147 (2019), 1467-1481.


\bibitem{WZ01}
J. Wu and X. Zou,
{\it Traveling wave fronts of reaction-diffusion systems with delay},
{J. Dynam. Differential Equations} 13 (2001), 651-687.

\bibitem{ywl19}
Y. Yang, C. Wu, Z. Li, {\it Forced waves and their asymptotics in a Lotka-Volterra cooperative model under climate change}, Appl. Math. Comput., 353 (2019), 254-264.



\bibitem{Z17}
T. Zhang,
{\it Minimal wave speed for a class of non-cooperative reaction-diffusion systems of three equations},
{J. Differential Equations} 262 (2017), 4724-4770.

\bibitem{ZJ17}
T. Zhang and Y. Jin,
{\it Traveling waves for a reaction-diffusion-advection predator-prey model},
{Nonlinear Analysis: Real World Applications} 36 (2017), 203-232.

\end{thebibliography}
\end{document}